\documentclass[12pt,english]{article}
\usepackage[margin=1.5in]{geometry}

\usepackage{mdwlist}
\usepackage{enumerate}
\usepackage{amssymb,amsbsy,latexsym}
\usepackage{amsmath}
\usepackage{graphics, subfigure, float}
\usepackage{fp, calc}
\usepackage{hyperref}
\usepackage{url}

\usepackage[T1]{fontenc} 
\usepackage{fourier}
\usepackage{bm}

\usepackage{amscd,amsthm}

\usepackage{pst-all}
\usepackage{pstricks-add}
\usepackage{pst-func}
\newpsobject{showgrid}{psgrid}{subgriddiv=1,griddots=10,gridlabels=6pt}
\usepackage{verbatim, comment}
\usepackage{datetime}

\newtheoremstyle{theorem}{1em}{1em}{\slshape}{0pt}{\bfseries}{.}{ }{}
\theoremstyle{theorem}
\newtheorem{theorem}{Theorem}

\newtheorem*{theorem*}{Theorem}
\newtheorem{corollary}[theorem]{Corollary}
\newtheorem*{corollary*}{Corollary}
\newtheorem{proposition}[theorem]{Proposition}
\newtheorem*{proposition*}{Proposition}
\newtheorem{lemma}[theorem]{Lemma}
\newtheorem{definition}[theorem]{Definition}

\newtheorem*{claim*}{Claim}
\theoremstyle{remark}

\newtheorem*{remark*}{Remark}

\providecommand{\setN}{\mathbb{N}}

\providecommand{\setR}{\mathbb{R}}

\newcommand{\E}{\mathop{\mathbb{E}}}

\newcommand{\Vol}{\mathrm{Vol}}

\renewcommand{\span}{\textrm{span}}

\newcommand{\eps}{\varepsilon}
\newcommand{\detplus}{\textstyle{\det_+}}
\renewcommand{\Tr}{\mathrm{Tr}}
\renewcommand{\Pr}{\mathop{\mathbb{P}}}

\usepackage{calrsfs} 
\DeclareMathAlphabet{\pazocal}{OMS}{zplm}{m}{n}

\makeatletter
\providecommand\@dotsep{5}
\def\listtodoname{List of Todos}
\def\listoftodos{\@starttoc{tdo}\listtodoname}
\makeatother

\title{Linear-size $\ell_1$ sparsifiers}
\author{Victor Reis\thanks{Microsoft Research, Redmond. Email: {\tt victorol@microsoft.com}.} \;\; and \; Thomas Rothvoss\thanks{University of Washington, Seattle. Email: {\tt rothvoss@uw.edu}. Supported by NSF grant 2318620 \emph{AF: SMALL: The Geometry of Integer Programming and Lattices}.}}
\date{}

\begin{document}

\maketitle

\begin{abstract} 
   We prove that for any matrix $A \in \setR^{m \times n}$ and any $\eps > 0$ there is a diagonal matrix $D \in \setR_{\geq 0}^{m \times m}$ with at most $O(\frac{n}{\varepsilon^2})$
  nonzero entries so that
  \[
  (1-\varepsilon) \|Ax\|_1 \leq \|DAx\|_1 \leq (1+\varepsilon)\|Ax\|_1 \quad \forall x \in \setR^n.
  \]
  In particular, for any zonotope $Z \subseteq \setR^{n}$ there exists a zonotope $Z' \subseteq \setR^{n}$ generated by at most $O(\frac{n}{\varepsilon^2})$ segments so that $(1-\varepsilon) Z \subseteq Z' \subseteq (1+\varepsilon) Z$. Previously, the best known bound was $O(\frac{n}{\varepsilon^2} \log n)$ due to Talagrand (1990). 
\end{abstract}


\section{Introduction}

A classical problem in the intersection of convex geometry and the design of fast algorithms is the
following: given a matrix $A \in \setR^{m \times n}$ and parameters $p \geq 1$ and $\varepsilon > 0$, can one
replace $A$ by a matrix $\widetilde{A}$ with few rows so that
\begin{equation} \label{eq:LpApproximation}
 (1-\varepsilon) \|Ax\|_p \leq  \|\widetilde{A}x\|_p \leq (1+\varepsilon) \|Ax\|_p \quad \forall x \in \setR^n.
\end{equation}
All existing methods 
 produce the matrix $\widetilde{A}$ by selecting and rescaling existing rows of $A$, i.e. $\widetilde{A} = DA$ where $D \in \setR_{\geq 0}^{m \times m}$ is a diagonal matrix with few nonzero entries. This may be preferred in the underlying applications, and we also restrict to this choice.

The setting of $p=1$ --- which is the focus of this manuscript --- has a natural geometric interpretation. For a compact convex set $K \subseteq \setR^n$, its \emph{support function} is the function $h_K : \setR^n \to \setR$ with
\[
  h_K(x) = \max_{y \in K} \left<y,x\right>.
\]
A matrix $A \in \setR^{m \times n}$ generates a \emph{zonotope} $Z = \{ y^\top A : y \in [-1,1]^m \}$ which is a bounded centrally symmetric polyhedron. In other words, $Z$ is the Minkowski sum of the $m$ segments $[-a_i,a_i]$
where $a_1,\ldots,a_m$ are the rows of $A$.
The support function of $Z$ is simply
\[
 h_Z(x) = \max_{y \in [-1,1]^m} \sum_{i=1}^m \left<y_ia_i,x\right> = \|Ax\|_1.
\]
Then our sparsification question \eqref{eq:LpApproximation} for $p = 1$ is equivalent to asking whether for any zonotope $Z \subseteq \setR^n$, there exists a zonotope $Z'$ with few segments so that $(1-\varepsilon) Z \subseteq Z' \subseteq (1+\varepsilon)Z$. Schechtman~\cite{Schechtman1987} proved that $O(\frac{n^2}{\varepsilon^2} \log(\frac{1}{\varepsilon}))$ segments suffice. This was later improved by Bourgain, Lindenstrauss and Milman \cite{BourgainLindenstraussMilman1989} to $O(\frac{n}{\varepsilon^2} \log(\frac{n}{\varepsilon}) \cdot (\log n)^2)$
and then by Talagrand~\cite{Talagrand1990} to $O(\frac{n}{\varepsilon^2} \log n)$ which remained the best known bound prior to this work. For the special case where $Z = B^n_2$ (which is a \textit{zonoid}), Figiel, Lindenstrauss and Milman~\cite{FigielLindenstraussMilman1977} proved that $O(\frac{n}{\varepsilon^2} \log(\tfrac{1}{\eps}))$ segments suffice, and Gordon~\cite{Gordon1985} improved this bound to $O(\frac{n}{\varepsilon^2})$ segments. More work has been done in the regime where $n$ is fixed and $\varepsilon \to 0$, see also the extensive discussion in \cite{BourgainLindenstraussMilman1989}.

A related question that has been studied extensively in the theoretical computer science community is how to replace an undirected graph $G = (V,E)$ on $n$ vertices by a weighted graph $G' = (V,E',w)$ with few edges that is a $(1+\eps)$-\textit{cut sparsifier}, i.e. for every set $S \subseteq V$, the weight of the cut is approximately preserved: $(1-\varepsilon) |\delta(S)| \leq w(\delta_{E'}(S)) \leq (1+\varepsilon) |\delta(S)|$. Motivated by designing faster algorithms for minimum $s$-$t$ cuts, Bencz\'ur and Karger~\cite{BenczurKargerSTOC96} proved that $O(\frac{n}{\varepsilon^2} \log n)$ edges suffice. The later works of Spielman and Teng~\cite{SpielmanTengSICOMP2011} and Spielman and Srivastava~\cite{SpielmanSrivastavaSICOMP2011} strengthened the notion of sparsification by showing that even the graph Laplacian can be approximately preserved using $O(\frac{n}{\varepsilon^2} (\log n)^7)$ and $O(\frac{n}{\varepsilon^2} \log n)$ edges, respectively. 
The breakthrough of Batson, Spielman and Srivastava \cite{BatsonSpielmanSrivastavaSTOC2009,BatsonSpielmanSrivastava2012} finally provided a bound that is  linear in $n$. In fact, \cite{BatsonSpielmanSrivastavaSTOC2009,BatsonSpielmanSrivastava2012} prove the more general statement that for any $A \in \setR^{m \times n}$ there is a diagonal matrix $D \in \setR_{\geq 0}^{m \times m}$  with support size $O(\frac{n}{\varepsilon^2})$ so that
\[
 (1-\varepsilon) \|Ax\|_2 \leq  \|DAx\|_2 \leq (1+\varepsilon) \|Ax\|_2 \quad \forall x \in \setR^n.
\]
This has a wide range of applications, such as an approximate John decomposition with only linearly many contact points, see Naor~\cite{SparseQuadraticFormsNaor2012}.
Most of the above results randomly sample rows (or edges) while fixing those that are deemed too important to be left to the randomness. Here the importance is determined for example using \emph{Lewis weights} (as in \cite{Talagrand1990,CohenPengSTOC2015}) or using \emph{effective resistances} (as in \cite{SpielmanSrivastavaSICOMP2011}). 
Crucially, the work of Batson, Spielman and Srivastava \cite{BatsonSpielmanSrivastavaSTOC2009,BatsonSpielmanSrivastava2012} provides a deterministic polynomial-time procedure using a potential function instead.
Now, let $G = (V,E)$ be a graph with $n = |V|$ vertices and $|E| = m$ edges. Create a matrix  $A \in \{ -1,0,1\}^{m \times n}$ which for any edge $e = \{ i,j\} \in E$ has a row $e_i - e_j$ (with an arbitrary orientation of the sign); this is a signed node-edge incidence matrix of the graph.
Then for any weight vector $w \in \setR_{\geq 0}^m$ and $S \subseteq [n]$, one has
\[
 \|\textrm{diag}(w) A \bm{1}_S\|_1 = w(\delta(S)).
\]
In other words, any $\ell_1$-sparsifier of the signed node-edge incidence matrix also gives a cut sparsifier. Andoni, Krauthgamer and Woodruff~\cite{AndoniKWGraphCuts2014} prove that for any $n$ and $\varepsilon > \frac{1}{\sqrt{n}}$ there is an $n$-vertex graph $G$ so that any $(1+\varepsilon)$-approximate cut sparsifier has at least $\Omega(\frac{n}{\varepsilon^2})$ edges. This also gives the same lower bound for the size of $\ell_1$-sparsifiers, at least when restricted to rescaling of existing rows.

We also mention a different type of sparsification that preserves the norm for pairs of points in a given discrete point set instead of all points in a subspace.
The by now classical result by Johnson and Lindenstrauss~\cite{Johnson1984ExtensionsOL} shows that
for any finite set $X \subseteq \setR^d$ there is a matrix $B$ with $O(\frac{1}{\varepsilon^2}\log(|X|))$ rows so that
\[
 (1-\varepsilon) \|x-y\|_2 \leq \|Bx - By\|_2 \leq (1+\varepsilon) \|x-y\|_2 \quad \forall x,y \in X.
\]
Much of the more recent work on this topic has focused on producing a matrix $B$ that is particularly sparse~\cite{DasguptaKumarSarlosSTOC2010,KaneNelsonJACM14}.  

More recently, there has been renewed interest in sparsification for functions beyond $\ell_p$ norms~\cite{JLLS2023,JLLS2024}.

\subsection{Our contribution}

Our main result is as follows:
\begin{theorem} \label{thm:L1Sparsifier}
For any matrix $A \in \setR^{m \times n}$ and any $0<\varepsilon \leq 1$, there is a diagonal matrix $D \in \setR_{\geq 0}^{m \times m}$ with at most $O(\frac{n}{\varepsilon^2})$
 nonzero entries so that
  \[
  (1-\varepsilon) \|Ax\|_1 \leq \|DAx\|_1 \leq (1+\varepsilon)\|Ax\|_1 \quad \forall x \in \setR^n.
  \]
\end{theorem}
As discussed above, this theorem has a geometric interpretation:
\begin{theorem} \label{thm:L1SparsifierV2}
For any zonotope $Z \subseteq \setR^{n}$ and any $0<\varepsilon \leq 1$, there exists a zonotope $Z'$ generated by at most $O(\frac{n}{\varepsilon^2})$ segments so that
  ${(1-\varepsilon) Z \subseteq Z' \subseteq (1+\varepsilon) Z}$. 
\end{theorem}
In fact, there is a third interpretation. Let $\ell_1$ be the infinite-dimensional
Banach space of all sequences $(x_k)_{k \in \setN}$
with $\|x\|_1 < \infty$, equipped with the $\| \cdot \|_1$-norm.
Moreover, let $\ell_1^N = (\setR^N,\| \cdot \|_1)$ which is a Banach space of dimension $N$.
\begin{theorem} \label{thm:L1SparsifierV3}
  Let $0<\varepsilon \leq 1$ and let $X$ be an $n$-dimensional subspace of $\ell_1$. Then there exists an $n$-dimensional
  subspace $Y$ of $\ell_1^N$ with $N \leq O(\frac{n}{\varepsilon^2})$ so that $d_{BM}(X,Y) \leq 1+\varepsilon$.
\end{theorem}
Here $d_{BM}(X,Y)$ denotes the \emph{Banach-Mazur distance} between $X$ and $Y$.
These results affirmatively answer a 1986 question of Schechtman~\cite[Problem 7]{Schechtman1987}; see also \cite{AIM2007FourierConvex}. We should also point out that our results do not come with a polynomial-time construction. This might be somewhat natural, as given two zonotopes $Z_1$ and $Z_2$, both specified by their generators, it is ${\bf coNP}$-complete to decide if $Z_1 \subseteq Z_2$~\cite{KulmburgAlthoff2021}. However, the matrix $D$ in Theorem~\ref{thm:L1Sparsifier} can be computed in time $2^m$ times a polynomial in the encoding length of $A$. The same holds true for the zonotope $Z'$ in Theorem~\ref{thm:L1SparsifierV2} where $A$ is the matrix whose rows generate $Z$.

Our framework also yields an alternative proof for the existence of linear-size spectral sparsifiers:
matching the bound of \cite{BatsonSpielmanSrivastavaSTOC2009,BatsonSpielmanSrivastava2012}; see also~\cite{ReisRothvoss2020} and~\cite{PaschalidisZhuang2024}.
\begin{theorem} \label{thm:L2Sparsifier}
For any matrix $A \in \setR^{m \times n}$ and $0<\varepsilon \leq 1$, there is a diagonal matrix $D \in \setR_{\geq 0}^{m \times m}$ with at most $O(\frac{n}{\varepsilon^2})$ nonzero entries so that
  \[
  (1-\varepsilon) \|Ax\|_2 \leq \|DAx\|_2 \leq (1+\varepsilon)\|Ax\|_2 \quad \forall x \in \setR^n.
  \]
  The weights can be computed in polynomial time.
\end{theorem}

\subsection{Overview}
In order to prove our results, we switch to the zonotope view and work towards a proof of Theorem~\ref{thm:L1SparsifierV2}; the equivalent statements of Theorem~\ref{thm:L1Sparsifier} and Theorem~\ref{thm:L1SparsifierV3} then follow. We fix a matrix $A \in \setR^{m \times n}$ with rows $a_1,\ldots,a_m$ and the corresponding zonotope $Z = \{ \sum_{i=1}^m y_ia_i : |y_i| \leq 1\; \forall i \in [m]\}$ where we assume that $Z$ is full-dimensional. For a vector $t \in \setR_{\geq 0}^m$ we also write $Z_t := \{ \sum_{i=1}^m y_ia_i : |y_i| \leq t_i \; \forall i \in [m] \}$ as the zonotope where the $i$th segment is scaled by $t_i$.

First, for compact convex sets $K,Q \subseteq \setR^n$, the \emph{Minkowski subtraction} 
can be defined as the inverse operation of the Minkowski addition, i.e. $K \ominus Q$ is the set of points
$x$ so that $x + Q \subseteq K$. Now consider the function $f : [-1,1]^m \to \setR_{\geq 0}$ defined by
\[
  f(x) := \Vol_n(Z \ominus Z_{\bm{1}+x})
\]
which gives the left-over volume when changing the segment lengths according to $x$. In particular $f(x) > 0$
implies that $Z_{\bm{1}+x} \subseteq Z$. A crucial insight is that the function $f$ is \emph{separately convex}, i.e. it is convex in each single variable. Moreover, the function inherits some natural scaling properties from the volume, in particular $f(\rho x) = \rho^n f(x)$ for all $x \in [-1,1]^m$ and $0 \leq \rho \leq 1$. Using \emph{hypercontractivity} we can prove that any function $f$ with these properties satisfies that
\begin{equation} \label{eq:Overview1}
  \Pr_{x \sim \{ -1,1\}^m}[f(x) > 0] \geq c^n
\end{equation}
where $c>0$ is a universal constant. The same bound will also hold when drawing $x \sim [-1,1]^m$.

The bound from \eqref{eq:Overview1} implies that the convex set
\[
 K := \big\{ x \in [-1,1]^m : Z_{\bm{1}+x} \subseteq (1+\varepsilon) Z \big\}
\]
of segment updates that would respect the upper inclusion is sufficiently large. In fact, for this to be true, 
the $\varepsilon$-slack would not even be needed. 
In order to also guarantee lower bounds in the form $(1-\varepsilon) Z \subseteq Z_{\bm{1}+x}$ we need to prove that the \emph{symmetrizer} $K \cap -K$ is still sufficiently large. While $K$ is highly asymmetric, all coordinate sections of $K$ still have measure at least $c^n$. With a careful convex geometric argument we prove in Section~\ref{sec:SizeOfSymmetrizer} that this property suffices to derive that $\Vol_m(K \cap -K) \geq 2^{-\Theta(m)}$ as long as $m \geq \Omega(\frac{n}{\varepsilon^2})$.

The remainder of the  argument is mostly standard. As $K \cap -K$ is large enough, we obtain that for some constant $s>0$, there is a vector $x \in s(K \cap -K)$ with $x_i = -1$ for at least a quarter of the coordinates. Then replacing $Z$ by $Z' := Z_{\bm{1}+x}$ reduces the number of segments by at least a quarter while $(1-s\varepsilon) Z \subseteq Z' \subseteq (1+s\varepsilon) Z$. We repeat the process until the target number of segments is reached; here the admissible value of $\varepsilon$ that makes the argument work increases geometrically.

\section{Preliminaries}

\paragraph{Convex geometry.} For a convex body $K \subseteq \setR^m$, the \emph{barycenter} is the point $\E_{x \sim K}[x]$.
We recall the following seminal result:

\begin{theorem}[Milman, Pajor~\cite{MilmanPajor2000}] \label{thm:MilmanPajorInequality} 
  For any convex body $K \subseteq \setR^m$ with barycenter $z$ one has $\Vol_m((K-z) \cap (-(K-z))) \geq 2^{-m}\Vol_m(K)$.
\end{theorem}

The following is a variant of \emph{Brunn's concavity principle} which can be derived either via the Brunn-Minkowski Theorem or via Steiner symmetrization, see \cite[Section~1.2]{AsymptoticGeometricAnalysis-Book2015} or \cite[Lemma 1.34]{rothvoss-asymptotic-convex-geometry}:
\begin{theorem}[Brunn's concavity principle] \label{thm:IntersectionConcavity} 
Let $K,L \subseteq \setR^m$ be convex bodies. 
Then the function  $F(x) := \Vol_m(K \cap (x+L))^{1/m}$ is concave on its support.
\end{theorem}

\paragraph{Discrepancy theory.}
The following is a minor modification of~\cite[Theorem 9]{ReisRothvossRSA23}:
\begin{theorem} \label{thm:FractionalColoring}
  For any $c > 0$ there is a $s := s(c) > 0$ so that the following holds:
  for any symmetric convex body $K \subseteq [-1,1]^m$ with $\Vol_m(K) \geq c^m$, there is an $x \in sK \cap [-1,1]^m$ with $|\{ j \in [m] : |x_j| = 1\}| \geq \frac{m}{2}$. Moreover, $x$ can be found in randomized polynomial time, given a separation oracle for $K$.
\end{theorem}
To be exact, \cite[Theorem 9]{ReisRothvossRSA23} uses Gaussian measure instead of volume.
But for any set $K \subseteq [-1,1]^m$ we have $\gamma_m(K)/\Vol_m(K) \in [(2\pi e)^{-m/2}, (2\pi)^{-m/2}]$ which only results in a change of the constant $c$.

\paragraph{Probability.}
We use the following variant of Hoeffding's inequality.
For the sake of completeness, its short proof can be found in Appendix~\ref{sec:AppendixA}.
\begin{lemma} \label{lem:cubeconc}
Let $v \in \setR^m$ and $x \sim [-1,1]^m$ uniformly. Then $\mathbb{P}[\langle v, x \rangle > \sqrt{\lambda} \|v\|_2] \le 2^{-2\lambda}$ for any $\lambda \ge 0$.
\end{lemma}
We also use H\"older's Inequality which is a generalization of Cauchy-Schwarz.
\begin{lemma}[H\"older's Inequality] \label{lem:Hoelder}
 Let $X,Y \in \setR$ be jointly distributed random variables and let $(p,q)$ be  H\"older conjugates, i.e. $p,q \geq 1$ and $\frac{1}{p} + \frac{1}{q} = 1$. Then one has $\E[|X \cdot Y|] \leq \E[|X|^p]^{1/p} \cdot \E[|Y|^q]^{1/q}$.
 \end{lemma}
  




\paragraph{Functions on the hypercube.}

For a function $f : \{ -1,1\}^m \to \setR$ on the hypercube we define its $p$-norm for $p \in [1,\infty)$
as $\|f\|_p := \E_{x \sim \{ -1,1\}^m}[|f(x)|^p]^{1/p}$.
Given a function $f : \{-1,1\}^m \to \mathbb{R}$ and $\rho \in [0,1]$, the \textit{noisy distribution} $N_\rho (x)$ at some $x \in \{-1,1\}^m$ is the random vector with independent coordinates and $i$-th coordinate $x_i$ with probability $\frac{1+\rho}{2}$, and $-x_i$ with probability $\frac{1-\rho}{2}$, so that $\E_{z \sim N_\rho (x)} [z] = \rho x$. The \textit{noisy operator} is defined as $T_\rho f (x) := \E_{z \sim N_\rho (x)} [f(z)]$. We will need Bonami's hypercontractivity theorem~\cite{Bonami1970}; see also
O'Donnell~\cite[Chapter~10.1]{ODonnell2014}.
\begin{theorem}\label{thm:hypercontractivity}
Let $f:\{-1,1\}^m\to\setR$, $1\le p\le 2$ and $0\le\rho\le\sqrt{ p-1}$. Then \[ \|T_\rho f\|_2 \le \|f\|_p.\]
\end{theorem}
We use the following lemma to relate different norms:
\begin{lemma} \label{lem:fp-vs-f2}
  Let $f : \{ -1,1\}^m \to \setR$ and $1 \leq p < 2$. Then $\|f\|_p \leq \alpha^{1/p-1/2}\|f\|_2$
  where $\alpha := \Pr_{x \sim \{ -1,1\}^m}[f(x) \neq 0]$ is the density of the support.
\end{lemma}
\begin{proof}
  In the following $x \sim \{ -1,1\}^m$.
 We can write
  \[
 \|f\|_p = \E[|f(x)|^p \cdot \bm{1}_{f(x) \neq 0}]^{1/p} \leq \E[|f(x)|^2]^{1/2} \cdot \alpha^{\frac{1}{p} \cdot \frac{2-p}{2}} = \|f\|_2 \cdot \alpha^{\frac{1}{p}-\frac{1}{2}}
\]
using H\"older's inequality (Lemma~\ref{lem:Hoelder}) with the conjugate pair $(\frac{2}{p},\frac{2}{2-p})$.
\end{proof}

\paragraph{Convex functions.}

Let $\Omega \subseteq \setR$ be a convex set. We say that a function $f : \Omega^m \to \setR$ is \emph{separately convex} if for any $t \in \Omega^m$ and $j \in [m]$, the function $s \mapsto f(t + se_j)$ is convex. We note that this does not imply that $f$ is convex. For example, all multilinear functions such as $f(x) = -x_1\cdot x_2$ are separately convex but not necessarily convex.
But separate convexity is still helpful: 
\begin{lemma}[Jensen's Inequality] \label{lem:JensenForSepConvexFcts}
  Let $f : \Omega^m \to \setR$ be a separately convex function and let $X \in \Omega^m$ be a random vector with independent coordinates. Then
  \[
    \E[f(X)] \geq f(\E[X])
  \]
\end{lemma}
\begin{proof}
  We prove this by induction over $m$. For $m = 1$ this is just convexity, so assume it holds for $m-1$ for some $m > 1$ and write $X = (X_1,\bar{X})$. Then 
  \[
   \E[f(X)] =  \E_{\bar{X}}[\E_{X_1}[f(X_1,\bar{X})]] \stackrel{\textrm{sep.}}{\geq} \E_{\bar{X}}[f(\E[X_1],\bar{X})] \stackrel{\textrm{induction}}{\geq} f(\E[X_1],\E[\bar{X}]) = f(\E[X])
  \]
  using independence and the fact that restrictions of $f$ are still separately convex.
\end{proof}
\section{Zonotopes and Minkowski subtraction}

The goal of this section is to derive a lower bound on the probability that a random zonotope $Z_{X}$ is contained in a given convex body $K$ where $X$ is any non-negative random vector with independent coordinates.

\subsection{Minkowski subtraction}

For two sets $K,L \subseteq \setR^n$, their \emph{Minkowski addition} is $K + L = \{ x + y \mid x \in K, y \in L\}$.
This operation has a conceptual inverse, first introduced by Hadwiger~\cite{Hadwiger1950} and independently by Pontryagin~\cite{Pontryagin1967V2}.
\begin{definition}
  For compact convex sets $K,L \subseteq \setR^n$ with $L$ nonempty, the \emph{Minkowski subtraction} is
  the set
  \[
   K \ominus L := \{ x \in \setR^n : x + L \subseteq K \}.
  \]
\end{definition}
We summarize a few of its properties, most of which can be found in the contemporary book of Schneider~\cite[Page 146+]{Schneider_2013}.
\begin{proposition} \label{prop:PropertiesOfSubtraction}
  Consider compact convex sets $K,L,M \subseteq \setR^n$ with $L,M$ nonempty.
  \begin{enumerate*}
  \item[(A)] One has $K \ominus L = \bigcap_{y \in L} (K-y)$.
  \item[(B)] $K \ominus L$ is convex.
  \item[(C)] $K \ominus L \neq \emptyset \iff L \subseteq K$ for symmetric $K, L$.
  \item[(D)] $(K \ominus L) \ominus M = K \ominus (L + M)$ (associativity).
  \item[(E)] $K \ominus \lambda K = (1-\lambda)K$ for $0 \leq \lambda \leq 1$.
  \end{enumerate*}
\end{proposition}
\begin{proof}
  One has $x \in K \ominus L$ if and only if $x + y \in K$ for every $y \in L$, so $K \ominus L = \bigcap_{y \in L} (K-y)$, and any intersection of convex sets must be convex itself. This proves (A) and (B).
  For (C) we use that if $L \subseteq K$ then $\bm{0} \in K \ominus L$ and conversely for symmetric $K,L$, $K \ominus L$ is also symmetric, so $\bm{0} \in K \ominus L$ if it is nonempty, which implies $L \subseteq K$.
  For (D) we can write
  \begin{eqnarray*}
    (K \ominus L) \ominus M &=& \{x \in \setR^n : (x + M) + L \subseteq K\} \\ &=& \{x \in \setR^n : x + (M+L) \subseteq K\} = K \ominus (L + M).
  \end{eqnarray*}
  Finally we prove (E).
  If $x \in (1-\lambda)K$, then $x + \lambda K \subseteq K$ by convexity, so $(1 - \lambda) K \subseteq K \ominus \lambda K$. If $K = \emptyset$ then clearly $K \ominus \lambda K = \emptyset$. Otherwise, suppose that $x \in K \ominus \lambda K$, so that $x + \lambda K \subseteq K$, and choose a point $x_0 \in K$. Define a sequence of points $x_{k+1} := x + \lambda x_k$ for $k \ge 0$ so that by induction $x_k \in K$ for all $k \in \mathbb{N}$. If $\lambda = 1$, then since $x_k = x_0 + kx \in K$ for all $k \in \mathbb{N}$ and $K$ is bounded, $x = 0$. Otherwise, $(1-\lambda)^{-1} x = \lim_{k \to \infty} x_k \in K$ since $K$ is closed, so $x \in (1-\lambda) K$.
\end{proof}
As a warning we note that other natural appearing properties of subtraction fail. In particular for non-empty compact convex sets $K,Q$ one has $(K \ominus Q) + Q \subseteq K$, but the inclusion can be strict. For
example for $K := [-1,1]^2$ and $Q := B_2^2$ one has $K \ominus Q = \{ \bm{0} \}$ and so $(K \ominus Q) + Q = Q \subset K$.

Here is a crucial property:
\begin{lemma} \label{lem:VolumeOfSubtractionIsConvexInS}
Let $K \subseteq \setR^n$ be a compact convex set and let $v \in \setR^n$. Let $f(s) := \Vol_n(K \ominus s[-v,v])$. Then the function $f$ is convex on $\setR_{\ge 0}$. 
\end{lemma}
\begin{proof}
  If $v=\bm{0}$, $f$ is constant. Otherwise, consider a vector $z \in v^{\perp}$
  and consider the line $z + \setR v$. Say that the length of the intersection with $K$ is $\ell(z)$. The same line intersected with $K \ominus s [-v,v]$
  has length $(\ell(z) - 2s\|v\|_2)_+$ where $(t)_+ = \max\{ t,0\}$, and for each fixed $z$, the function $s \mapsto (\ell(z)-2s\|v\|_2)_+$ is convex on $\setR_{\ge 0}$.
  Then
  \[
 f(s) = \Vol_n(K \ominus s[-v,v]) = \int_{v^{\perp}} (\ell(z)-2s\|v\|_2)_+ dz
\]
is an average of convex functions which is again convex on $[0,\infty)$.
\end{proof}

\subsection{Minkowski subtraction and zonotopes}

As before, for a matrix $A \in \setR^{m \times n}$ with rows $a_1, \dots, a_m$
and $t \in \setR_{\geq 0}^m$ we abbreviate the zonotope $Z_t = \{ \sum_{i=1}^m y_ia_i : |y_i| \leq t_i \; \forall i \in [m]\}$.
We can generalize Lemma~\ref{lem:VolumeOfSubtractionIsConvexInS} from subtracting segments to subtracting zonotopes:
\begin{lemma} \label{lem:VolOfMinkSubtractionIsSepConvex}
  Let $K \subseteq \setR^n$ be a convex body and $Z = \sum_{j \in [m]} [-a_j, a_j]\subseteq \setR^n$ be a zonotope.
  Then $f : \setR_{\ge 0}^m \to \setR_{\geq 0}$ with $f(t) := \Vol_n(K \ominus Z_t)$ is separately convex.
\end{lemma}
\begin{proof}
  For any $k \in [m]$, we have
  \[
    K \ominus Z_t = \Big(K \ominus \sum_{j \in [m]\setminus \{k\}} t_j [-a_j, a_j]\Big) \ominus t_k [-a_k,a_k]
  \]
  by Proposition~\ref{prop:PropertiesOfSubtraction}(D), so the claim follows from Lemma~\ref{lem:VolumeOfSubtractionIsConvexInS}.
\end{proof}

Now we have everything in place to conclude an important insight. 
\begin{theorem}\label{thm:RandomZonotopeInclusion}
  Let $K \subseteq \setR^n$ be a symmetric convex body and $Z \subseteq \setR^n$ be a zonotope generated by $m$ segments.
  Let $X \in \setR_{\ge 0}^m$ be a random vector with independent coordinates. Then
  \[
    \mathbb{P}[Z_X \subseteq K] \geq \frac{\Vol_n(K \ominus Z_{\E[X]})}{\Vol_n(K)}
  \]
\end{theorem}
\begin{proof}
  We abbreviate $f(t) := \Vol_n(K \ominus Z_t)$ for $t \in \setR_{\geq 0}^m$ which is a separately convex function by Lemma~\ref{lem:VolOfMinkSubtractionIsSepConvex}. Then
  \begin{eqnarray*}
\mathbb{P}[Z_X \subseteq K] \cdot \Vol_n(K) &\geq& \mathbb{P}[\Vol_n(K \ominus Z_X)>0] \cdot \Vol_n(K)  \\ &\geq& \E[\Vol_n(K \ominus Z_X)] \\ &\geq& \Vol_n(K \ominus Z_{\E[X]}),
  \end{eqnarray*}
  using Lemma~\ref{lem:JensenForSepConvexFcts}. Moreover
 $\Vol_n(K \ominus Z_X)>0$ implies that $K \ominus Z_X \neq \emptyset \iff Z_X \subseteq K$. Rearranging  finishes the proof.
\end{proof}
Let us summarize that we learned so far. If we apply Theorem~\ref{thm:RandomZonotopeInclusion} with $K := (1+\varepsilon)Z$, we obtain
\[
  \Pr_{x \sim [-1,1]^m}[Z_{\bm{1}+x} \subseteq (1+\varepsilon) Z] \geq \frac{\Vol_n((1+\varepsilon)Z \ominus Z)}{\Vol_n((1+\varepsilon) Z)} = \Big(\frac{\varepsilon}{1+\varepsilon}\Big)^n
\]
Together with the framework that we develop in Section~\ref{sec:SizeOfSymmetrizer} and Section~\ref{sec:ProofOfLinSizeSparsifier} this probability would suffice for $\ell_1$-sparsifiers of size $O(\frac{n}{\varepsilon^2} \log(\frac{1}{\varepsilon}))$. In fact, for this bound the reader may skip Section~\ref{sec:InclusionOfRandZon} entirely.
However, next we will prove a stronger bound of $\mathbb{P}_{x \sim [-1,1]^m}[Z_{\bm{1}+x} \subseteq Z] \geq c^n$ which will then lead to the tight bound of $O(\frac{n}{\varepsilon^2})$ promised in Theorems~\ref{thm:L1Sparsifier}, \ref{thm:L1SparsifierV2} and \ref{thm:L1SparsifierV3}.

\section{Inclusion of random zonotopes\label{sec:InclusionOfRandZon}}

The proof will rely on the following crucial lemma:

\begin{lemma}\label{lem:hyperhomog}
Let $f : [-1,1]^m \to \mathbb{R}_{\ge 0}$ be separately convex and not identically zero so that $f(\rho x) = \rho^n f(x)$ for all $x \in \{-1,1\}^m$ and $0 \leq \rho \leq 1$. Then \[ \Pr_{x \sim \{-1,1\}^m } [f(x) > 0] \ge e^{-2n}. \]
\end{lemma}

\begin{proof}
  Let $\alpha := \Pr_{x \sim \{ -1,1\}^m}[f(x) > 0]$ be the parameter that we aim to lowerbound.
  Since $f$ is not identically zero, take $x^* \in [-1,1]^m$ with $f(x^*) > 0$. As $x^*$ lies in the hypercube, there is a product distribution $\pazocal{D}$ supported on the vertices $\{ -1,1\}^m$ so that $\E_{z \sim \pazocal{D}}[z] = x^*$. Then by separate convexity and Lemma~\ref{lem:JensenForSepConvexFcts} we have $\E_{z \sim \pazocal{D}}[f(z)] \ge f(x^*) > 0$. Hence we may conclude that $f$ is not identically zero on the hypercube vertices $\{-1,1\}^m$.

  Let $\rho \in (0,1)$ be a parameter that we determine later and let $p \in (1,2)$ be the
  unique parameter so that $\rho = \sqrt{p-1}$ as to satisfy Theorem~\ref{thm:hypercontractivity}.
  First we note that for any $x \in \{-1,1\}^m$, by separate convexity and Lemma~\ref{lem:JensenForSepConvexFcts} we have  
  \begin{equation} \label{eq:TrhoFvsF}
    T_\rho f(x) = \E_{z \sim N_\rho (x)} [f(z)] \ge f\Big(\E_{z \sim N_\rho (x)} [z]\Big) = f(\rho x) \ge \rho^n f(x).
  \end{equation}
  Then applying hypercontractivity (Theorem~\ref{thm:hypercontractivity}) and Lemma~\ref{lem:fp-vs-f2}
  and using non-negativity of $f$ we obtain
  \[
\rho^n\|f\|_2  \stackrel{\eqref{eq:TrhoFvsF}}{\leq} \|T_{\rho}f\|_2 \leq  \|f\|_p \leq \|f\|_2 \cdot \alpha^{\frac{1}{p}-\frac{1}{2}}
  \]
  
  Since $f$ is not identically zero on $\{-1,1\}^m$, it follows that $\|f\|_2 > 0$. As $p = 1+\rho^2$,
  rearranging for $\alpha$ and taking the limit $\rho \to 1$ gives
  \[
   \alpha \geq \lim_{\rho \to 1} \rho^{\frac{2(1+\rho^2)}{1-\rho^2} \cdot n} = e^{-2n}. \qedhere
  \]
\end{proof}
The result also applies to $x \sim [-1,1]^m$.
\begin{corollary}\label{cor:hyperhomogCont}
  Let $f : [-1,1]^m \to \mathbb{R}_{\ge 0}$ be separately convex so that $f(\rho x) = \rho^n f(x)$ for all $x \in [-1,1]^m$ and $0 \leq \rho \leq 1$.
Assume that $f$ is strictly positive on the orthant $\setR_{<0}^m$.
Then
\[ \Pr_{x \sim [-1,1]^m } [f(x) > 0] \ge e^{-2n}. \]
\end{corollary}
\begin{proof}
  We can sample each coordinate $x_i$ by first sampling its absolute value $\sigma_i \sim [0,1]$ and then independently sampling $\varepsilon_i \sim \{-1,1\}$, so that $x_i=\varepsilon_i \sigma_i$. After we sampled
  $\sigma$, we define $g : [-1,1]^m \to \setR_{\geq 0}$ with $g(y) := f( (y_i \cdot \sigma_i)_{i \in [m]})$. Then $g$ is still separately convex with $g(\rho y) = \rho^n g(y)$ for all $y \in [-1,1]^m$ and $0 \leq \rho \leq 1$. Also, by assumption $g$ is not identically zero almost surely.
  Hence for any conditioning $\sigma$ we have $\Pr[f(x) > 0 \mid \sigma] = \Pr_{\varepsilon \sim \{ -1,1\}^m}[g(\varepsilon) > 0] > e^{-2n}$ by Lemma~\ref{lem:hyperhomog} which gives the claim.
\end{proof}

\subsection{The main theorem}

Now we are ready to show the main theorem of this section.

\begin{theorem} \label{thm:UniformRandomZonotopeInclusion}
  Let $Z \subseteq \setR^n$ be a zonotope generated by the rows of $A \in \setR^{m \times n}$.  Then
  \[
    \Pr_{x \sim [-1,1]^m}[Z_{\bm{1}+x} \subseteq Z] \geq e^{-2n}.
  \]
\end{theorem}

\begin{proof}
Let $V:=\span\{a_1,\ldots,a_m\}$ and $d:=\dim V$. Assume without loss of generality that $d > 0$. For $x\in[-1,1]^m$, define
\[
f(x):=\Vol_V(Z\ominus Z_{\bm{1}+x}).
\]
By Lemma~\ref{lem:VolOfMinkSubtractionIsSepConvex}, $f$ is separately convex.
For $x \in [-1,1]^m$ and $\rho\in [0,1]$, we have
\[
Z_{\bm{1}+\rho x}
=(1-\rho)Z+\rho Z_{\bm{1}+x}.
\]
Using Prop~\ref{prop:PropertiesOfSubtraction}(D, E) we may write
\[
\begin{aligned}
Z \ominus Z_{\bm{1}+\rho x}
  &= Z \ominus ((1-\rho)Z+\rho Z_{\bm{1}+x}) \\
  &= (Z \ominus (1-\rho)Z) \ominus \rho Z_{\bm{1}+x} \\
  &= \rho Z \ominus \rho Z_{\bm{1}+x} \\
  &= \rho (Z \ominus Z_{\bm{1}+x}),
\end{aligned}
\]
and so $f(\rho x)=\rho^d f(x)$. Moreover $f$ is positive on the strictly negative orthant, 
so Cor~\ref{cor:hyperhomogCont}  yields
\[
\Pr_{x\sim[-1,1]^m}[f(x)>0]
\geq e^{-2d}\geq e^{-2n}.
\]
Finally, $f(x)>0$ implies $Z_{\bm{1}+x}\subseteq Z$ by
Prop~\ref{prop:PropertiesOfSubtraction}(C). 
\end{proof}

\section{The size of the symmetrizer\label{sec:SizeOfSymmetrizer}}

Consider a set $K \subseteq \setR^m$ and a subset $S \subseteq [m]$ of coordinates. We denote by $K_S := \{ x' \in \setR^S : (x',\mathbf{0}) \in K\}$ the intersection of $K$ with the coordinate subspace corresponding to $S$, where $\mathbf{0} \in \setR^{[m] \setminus S}$.
We summarize what we can derive from the last section:
\begin{corollary} \label{thm:VolumetricLowerBoundSections}
  Let $Z \subseteq \setR^n$ be a zonotope generated by the rows of $A \in \setR^{m \times n}$ and let $\eps > 0$. Then the set
  \[ K := \{x \in [-1,1]^m : Z_{\mathbf{1}+x} \subseteq (1+\varepsilon)Z\}\] contains $[-\varepsilon,\varepsilon]^m$, is convex and $\Vol_S (K_S) \ge c^n \cdot 2^{|S|}$ for every $S \subseteq [m]$ where $c>0$ is a universal constant.
\end{corollary}

\begin{proof} 
  We may equivalently write
  \begin{equation} \label{eq:KinInequalityForm}
    K = \Big\{x \in [-1,1]^m : \sum_{i=1}^m x_i \cdot |\langle a_i, y \rangle| \le \eps \sum_{i =1}^m |\langle a_i, y \rangle| \ \forall y \in \setR^n\Big\},
 \end{equation}
 which means that $K \supseteq [-\eps,\eps]^m$ is the intersection of halfspaces and therefore convex. 
 For fixed $S \subseteq [m]$ we have
 \[
   \frac{\Vol_S(K_S)}{2^{|S|}} \geq \Pr_{x_S \sim [-1,1]^S}[(Z_{\bm{1}_S+x_S}+Z_{\bm{1}_{\bar{S}}}) \subseteq Z] =
   \Pr_{x_S \sim [-1,1]^S}[Z_{\bm{1}_S+x_S} \subseteq Z_{\bm{1}_S}] \geq e^{-2n}
 \]
 by Theorem~\ref{thm:UniformRandomZonotopeInclusion}\footnote{The reader may note that the $\varepsilon$-slack is only needed for the inclusion of the cube $[-\varepsilon,\varepsilon]^m$, not to lower bound the measure of sections.}. Here we used the fact that
 for compact convex sets $A,B,C \subseteq \setR^n$ with $C$ non-empty one has $A \subseteq B \Leftrightarrow A+C \subseteq B+C$.

\end{proof}
For our overall proof strategy we need to prove that $K \cap -K$ is still sufficiently large. In fact, this turns out to be true using exactly the
properties that we obtained in Theorem~\ref{thm:VolumetricLowerBoundSections}. 
The main result of this section will be the following:
\begin{theorem} \label{thm:LowerboundOnVolumeForKSymmetrizer}
Let $p, \eps \in (0, 1/2]$ and $m$ be a positive integer so that $m \ge \frac{\log_2 (1/p)}{\eps^2}$. Let $K \subseteq [-1,1]^m$ be a convex body so that $[-\eps,\eps]^m \subseteq K$ and, for every $S \subseteq [m]$, $\Vol_S (K_S) \ge p2^{|S|}$. Then $\Vol_m (K \cap -K) \ge 2^{-5m}$.
\end{theorem}

The following will be helpful to certify that the symmetrizer is large:
\begin{proposition} \label{prop:reflectionVolume}
  Let $K \subseteq \setR^m$ be a convex body. Let $f_K(z) := \Vol_m(K \cap (2z-K))$ denote the volume of the intersection of $K$ with its reflection around $z$, supported on $z \in K$.
  \begin{enumerate}
  \item[(A)] There exists $z^* \in K$ so that $f_K(z^*) \ge 2^{-m} \Vol_m (K)$. In fact, the barycenter of $K$ is a valid choice for $z^*$. 
  \item[(B)] For any $z \in K \cap -K$ one has $\Vol_m (K \cap -K) \ge 2^{-m} f_K(z)$.
  \item[(C)] For any $t > 0$, the set $\{z \in K: f_K(z) \ge t\}$ is convex.
  \end{enumerate}
\end{proposition}
\begin{proof}
{\bf (A).}  From the Milman-Pajor Theorem (Theorem~\ref{thm:MilmanPajorInequality}) we immediately know that $f_K(\E_{z \sim K}[z]) \geq 2^{-m}\Vol_m(K)$. But there is a direct proof as well, which we include:
  Sample $x \sim K$ uniformly at random. Note that $\mathbb{P}_{x \sim K} [x \in 2z-K] = \frac{f_K(z)}{\Vol_m(K)}$ for any $z \in K$. In particular, sampling $z \sim K$ uniformly at random independently,
\[ \frac{\E_{z \sim K} [f_K(z)]}{\Vol_m (K)} = \E_{z \sim K} \mathbb{P}_{x \sim K} [x \in 2z-K] = \E_{z \sim K} \E_{x \sim K} [\mathbf{1}_{\{x \in 2z-K\}}] = \E_{x \sim K} \mathbb{P}_{z\sim K} [x \in 2z-K].\]
Since $\mathbb{P}_{z\sim K} [x \in 2z-K] = \mathbb{P}_{z\sim K} [z \in (x+K)/2] = \frac{\Vol_m((x+K)/2)}{\Vol_m(K)} = 2^{-m}$, it follows that $\E_{z \sim K} [f_K(z)] = 2^{-m} \Vol_m (K)$ and some $z^* \in K$ satisfies $f_K(z^*) \ge 2^{-m} \Vol_m (K)$.

{\bf (B).} Let $Q = \frac{1}{2} (K \cap (2z-K) - z)$ and note that $\Vol_m (Q) = 2^{-m} f_K(z)$. It remains to check that $Q \subseteq K \cap -K$. Indeed, any $y \in Q$ may be written as $(x-z)/2$ for some $x \in K \cap (2z-K)$; we have $(x-z)/2 \in K$ by convexity since $x, -z \in K$, and also $(x-z)/2 = (x-2z+z)/2 \in -K$ by convexity since $z, x-2z \in -K$.

{\bf (C).} By Theorem~\ref{thm:IntersectionConcavity}, the function $f_K(z)^{1/m}$ is concave on its support. Hence the superlevel sets of $f_K(z)^{1/m}$ are convex. Then the same holds true for the superlevel sets of $f_K(z)$.
\end{proof}

In order to reach a contradiction from a separating hyperplane, we need the following key technical lemma:

\begin{lemma} \label{lem:VolumeEllOneCap} Let $p \in (0,1/2]$ and $\lambda = \sqrt{\log_2 (1/p)} \ge 1$. Let $K \subseteq [-1,1]^m$ be a convex body so that $\mathbf{0} \in K$ and for every $S \subseteq [m]$, $\Vol_S (K_S) \ge p 2^{|S|}$. Then for every $a \in \setR^m$, \[K_a := K \cap \Big\{x \in [-1,1]^m : \langle a, x \rangle \le \lambda/\sqrt{m} \cdot \|a\|_1 \Big\}\] has volume $\Vol_m (K_a) \ge p 2^{-2m}$.
\end{lemma}
\begin{proof}
 If $m \le 4\lambda^2$ then $\tfrac{1}{2} K \subseteq K_a$ as $\langle a,x \rangle \le \tfrac{1}{2} \|a\|_1 \le \lambda/\sqrt{m} \cdot \|a\|_1$ for all $x \in \tfrac{1}{2} [-1,1]^m$, so that $\Vol_m (K_a) \ge 2^{-m} \Vol_m(K) \ge p > p 2^{-2m}$.
 Assume $m > 4 \lambda^2$ and reorder coordinates so that $|a_1| \ge |a_2| \ge \cdots \ge |a_m|$.
Define a sequence of indices by $m_1=m$ and $m_{k+1}=\lfloor m_k/2\rfloor$ for $k \ge 1$, and let $\ell$ be the first index such that $m_\ell\le 4\lambda^2$; in particular $m_k > 4\lambda^2 \ge 4$ for $k < \ell$. Partition $[m]=I_\ell \sqcup I_{\ell-1}\sqcup\cdots\sqcup I_1$ where for $1\le k<\ell$, $I_k:=[m_k]\setminus [m_{k+1}]$, and $I_\ell := [m_\ell]$. For each $1\le k<\ell$, let
\[
    Q_k:=\Big\{x_{I_k} \in K_{I_k}:\langle a_{I_k},x_{I_k}\rangle \le \lambda\,\|a_{I_k}\|_2\Big\}.
\]
By Lemma~\ref{lem:cubeconc}, we have
$
    \Vol_{I_k}(Q_k) \ge \Vol_{I_k}(K_{I_k}) - p^2 2^{|I_k|} \ge p2^{|I_k|}-p^2 2^{|I_k|} \ge \tfrac{p}{2} 2^{|I_k|}.
$

For $I_\ell$ we set $Q_\ell :=  \Big\{ x_{I_\ell} \in K_{I_\ell}: \langle a_{I_\ell},x_{I_\ell} \rangle \le \lambda/\sqrt{m_\ell}\cdot \|a_{I_\ell}\|_1 \Big\}$ which has volume at least $p$ as $|I_\ell| \le 4\lambda^2$, so that $\tfrac{1}{2} K_{I_\ell} \subseteq Q_\ell$ as above.

Now define the Cartesian product
\[
        P := \Big\{x \in \setR^m : x_{I_\ell} \in 2^{-(\ell - 1)} Q_\ell \ \text{ and } \ x_{I_k}\in 2^{-k} Q_k\ \text{ for } k\in [1,\ell)\Big\}.
\]
It remains to show the following:
\begin{claim*}
$P \subseteq K_a$ and $\Vol_m (P) \ge p 2^{-2m}.$
\end{claim*}

The chosen coefficients satisfy $2^{-(\ell-1)}+\sum_{k\in [1,\ell)}2^{-k}=1$ so $P\subseteq K$ by convexity.
For each $1\le k<\ell$, every coordinate $a_i$ with  $i \in I_k = [m_k] \setminus [m_{k+1}]$ has absolute value at most $\frac{\|a_{[m_{k+1}]}\|_1}{m_{k+1}}$, so that $\|a_{I_k}\|_2^2 \le \|a_{I_k}\|_\infty\|a_{I_k}\|_1 \le \frac{\|a_{[m_{k+1}]}\|_1\|a_{I_k}\|_1}{m_{k+1}}$. We may bound
\[
\begin{aligned}
       2^{-k} \lambda \|a_{I_k}\|_2  &\le \frac{2^{-k}\lambda}{\sqrt{m_{k+1}}}  \sqrt{\|a_{[m_{k+1}]}\|_1\|a_{I_k}\|_1} \\ &\le \frac{2^{-(k-1)}\lambda}{\sqrt{m_k}}\|a_{[m_k]}\|_1 - \frac{2^{-k}\lambda}{\sqrt{m_{k+1}}} \|a_{[m_{k+1}]}\|_1,
\end{aligned}
\]
where we use $2\sqrt{m_{k+1}/m_k} \ge 2 \sqrt{\tfrac{2}{5}} > \frac{5}{4}$ and $\sqrt{uv} \le \frac{5}{4} (u+v) - u$ for $u,v \ge 0$.

Thus $P \subseteq K_a$, as for any $x\in P$ we indeed have, by telescoping,
 \begin{eqnarray*}
     \left<a,x\right> &=& \sum_{k=1}^{\ell-1} 
\left<a_{I_k},x_{I_k}\right> +
\left<a_{I_{\ell}},x_{I_{\ell}}\right> \\
     &\leq& \sum_{k=1}^{\ell-1} 2^{-k} \lambda \|a_{I_k}\|_2 + \frac{2^{-(\ell-1)} \lambda}{\sqrt{m_\ell}}\|a_{I_{\ell}}\|_1
  \\   &\leq& \frac{\lambda}{\sqrt{m_1}} \|a_{[m_1]}\|_1 =\lambda/\sqrt{m} \cdot \|a\|_1.
   \end{eqnarray*}
   It remains to lower bound $\Vol_m (P)$. The set $P$ arises from the sets $Q_k$ by scaling, so
\[
\begin{aligned}
        \Vol_m(P)
        &=
        \Vol_{I_\ell}(2^{-(\ell-1)} Q_\ell)
        \prod_{k=1}^{\ell-1}\Vol_{I_k}( 2^{-k} Q_k) \\
         & \ge 2^{-(\ell-1) m_\ell} \cdot p \cdot 
        \prod_{k=1}^{\ell-1} 2^{-k |I_k|}
        \cdot \tfrac{p}{2} 2^{|I_k|} \\
        &= 2^{-\sum_{k=1}^{\ell-1} m_k}\cdot 
        p \cdot 2^{m-m_\ell} \cdot (p/2)^{\ell-1} \\
        &=
        p \cdot 2^{m-\sum_{k=1}^{\ell} m_k} \cdot (p/2)^{\ell-1} \\
        & \ge p 2^{-m} \cdot 2^{-m} = p 2^{-2m},
        \end{aligned}
\]
where we use $\sum_{k=1}^\ell m_k < 2m$ as $m_{k+1} \le m_k /2$ for $k \ge 1$, and $(p/2)^{\ell-1} \ge 2^{-m}$ as $\ell-1 \le \lceil \log_2 (m/(4\lambda^2)) \rceil \le m/(2\lambda^2)$ and so $(2/p)^{\ell-1} = 2^{(\ell-1)(\lambda^2+1)} \le 2^{m \cdot \frac{\lambda^2+1}{2\lambda^2}} \le 2^m$.
\end{proof}

Now we have all the ingredients to show the main theorem of this section.

\begin{proof}[Proof of Theorem~\ref{thm:LowerboundOnVolumeForKSymmetrizer}]
  Let $K$ be a convex body with $\varepsilon B_{\infty}^m \subseteq K \subseteq B_{\infty}^m$ so that all coordinate sections of $K$ have relative volume of at least $p$.
  Our goal is to prove that $\Vol_m(K \cap -K)$ is large. We abbreviate $r := \sqrt{\frac{\log_2(1/p)}{m}}$ and prove the following: \\
  {\bf Claim I.} \emph{There is a $z^* \in K$ with $f_K(z^*) \geq p2^{-3m}$ and $\|z^*\|_{\infty} \leq r$.} \\
  {\bf Proof of Claim I.} Let $S := \{ z \in K \mid f_K(z) \geq p2^{-3m}\}$ which by Prop~\ref{prop:reflectionVolume}(C) is a convex set. Suppose for the sake of contradiction that there is no such $z^*$. That means the two convex sets $S$ and $r B_{\infty}^m$ are disjoint. Then by the separating hyperplane theorem, there is a hyperplane with normal vector $a \in \setR^m$ that separates the two sets. After choosing an orientation for $a$ we have
  \[
 \forall z \in S:  \;\; \left<a,z\right> > \max_{x \in rB_{\infty}^m} \left<a,x\right> = r \|a\|_{1}.
  \]
 Now, let $z^*$ be the barycenter of $K_a = K \cap \{ x \in [-1,1]^m : \left<a,x\right> \leq r\|a\|_1\}$. Then
  \[
   f_K(z^*) \geq f_{K_a}(z^*) \stackrel{\textrm{Prop~\ref{prop:reflectionVolume}}(A)}{\geq} 2^{-m}\Vol_m(K_a) \stackrel{(*)}{\geq} 2^{-m} \cdot p 2^{-2m} = p2^{-3m}.
 \]
 Here we apply Lemma~\ref{lem:VolumeEllOneCap} in $(*)$.
 Thus $z^* \in S$ while $\langle a, z^*\rangle \le r\|a\|_1$ which is a contradiction. \qed

 Now we conclude the main proof. By the assumption $m \ge \frac{\log_2(1/p)}{\eps^2}$ we have $\|z^*\|_\infty \le r \le \eps$. Since $[-\eps,\eps]^m \subseteq K$, we obtain $z^* \in K \cap -K$ and Prop~\ref{prop:reflectionVolume}(B) yields
 \[
   \Vol_m (K \cap -K) \ge 2^{-m} f_K (z^*) \ge p 2^{-4m} > 2^{-5m},
 \]
 as $p \ge 2^{-m\eps^2} \ge 2^{-m/4}$.
\end{proof}

\section{Proof of the main result} \label{sec:ProofOfLinSizeSparsifier}
We finally prove Theorem~\ref{thm:L1SparsifierV2} which also implies
the statements of Theorem~\ref{thm:L1Sparsifier} and Theorem~\ref{thm:L1SparsifierV3}. First we summarize what we learned from Section~\ref{sec:InclusionOfRandZon} and Section~\ref{sec:SizeOfSymmetrizer}.

\begin{corollary} \label{cor:VolumeSymmetricL1}
  There is a universal constant $c_0>0$ so that the following holds:
  Let $Z \subseteq \setR^n$ be a zonotope generated by the rows of $A \in \setR^{m \times n}$. For $\eps > 0$ and $m \ge \frac{c_0n}{\eps^2}$ the set
  \[
   Q = \{ x \in [-1,1]^m \mid (1-\varepsilon) Z \subseteq Z_{\bm{1}+x} \subseteq (1+\varepsilon) Z \}
  \]
  is convex and symmetric, and has volume $\Vol_m (Q) \ge 2^{-5m}$.
\end{corollary}
\begin{proof}
As in \eqref{eq:KinInequalityForm}, we can write $Q$ in inequality form as  
\[
  Q = \Big\{x \in [-1,1]^m : \Big|\sum_{i=1}^m x_i \cdot |\langle a_i, y \rangle|\Big| \le \eps \sum_{i =1}^m |\langle a_i, y \rangle| \ \forall y \in \setR^n\Big\},\]
from which symmetry and convexity are immediate.
Note that $Q$ is precisely $K \cap -K$ for the convex set $K$ defined in Theorem~\ref{thm:VolumetricLowerBoundSections}, which by that theorem satisfies $[-\eps,\eps]^m \subseteq K$ and $\Vol_S (K_S) \ge c^n 2^{|S|}$ for every $S \subseteq [m]$. For $c_0 \geq \log_2(1/c)$ we have  $m \ge \frac{c_0n}{\varepsilon^2} = \frac{\log_2(c^{-n})}{\eps^2}$ and thus by Theorem~\ref{thm:LowerboundOnVolumeForKSymmetrizer} we have $\Vol_m(Q) = \Vol_m(K \cap -K) \geq 2^{-5m}$.
\end{proof}

Now we prove Theorem~\ref{thm:L1SparsifierV2} which we restate in a slightly expanded form:
\begin{theorem} \label{thm:L1SparsifierV2rephrased}
Let $Z \subseteq \setR^{n}$ be a zonotope generated by the rows of $A \in \setR^{m \times n}$ and let $\varepsilon >0$. Then there exists a weight vector $w \in \setR_{\geq 0}^m$ with $|\textrm{supp}(w)| \leq O(\frac{n}{\varepsilon^2})$ so that
$(1-\varepsilon) Z \subseteq Z_w \subseteq (1+\varepsilon) Z$.
Moreover, the weight vector $w$ can be computed in randomized time $2^m$ times a polynomial in the encoding length of $A$.
\end{theorem}
\begin{proof}
  The proof strategy is to iteratively reduce the number of used segments via Theorem~\ref{thm:FractionalColoring} until it reaches a target of $M := C \cdot \frac{n}{\eps^2}$ where $C>0$ is a large enough constant.
  Starting with $w^{(0)} := \bm{1}$, we construct a sequence of vectors $w^{(t)} \in \setR^{m}_{\geq 0}$ for $t \in \{0, \dots, T\}$ supported on $S_t := \textrm{supp}(w^{(t)})$ with $m_t := |S_t|$ so that $m_t \le \tfrac{3}{4} m_{t-1}$ for each $t \in [T]$ and $T$ is the first index with  $m_T \le M$.
  Set $\eps_t$ so that $m_t = \frac{c_0n}{\eps_t^2}$. We note that $\varepsilon_{t-1} \leq \sqrt{3/4} \varepsilon_t$ for all $t \in \{ 1,\ldots,T\}$. Moreover $m_{T-1} > M$ and so $\varepsilon_{T-1} \leq \sqrt{c_0/C} \cdot \varepsilon$.
  For vectors $a,b \in \setR^m$ we denote by $a \odot b \in \setR^m$ the coordinate-wise product, i.e. $(a \odot b)_i := a_i \cdot b_i$. Consider the set
  \[ Q_{\eps_t} := \Big\{x \in [-1,1]^{S_t} : (1-\varepsilon_t) Z_{w^{(t)}} \subseteq Z_{(\bm{1}+x) \odot w^{(t)}} \subseteq (1+\varepsilon_t) Z_{w^{(t)}} \Big\},
  \]
  which by Corollary~\ref{cor:VolumeSymmetricL1} is a symmetric convex body
  of volume at least $2^{-5m_t}$. Then applying Theorem~\ref{thm:FractionalColoring} to $Q_{\varepsilon_t}$ with $c := 2^{-5}$ yields a vector $x^{(t)} \in sQ_{\eps_t} \cap [-1,1]^{S_t}$ with $|\{ i : |x^{(t)}_i| = 1\}| \ge \tfrac{1}{2} m_t$ where $s>0$ is some constant. After possibly flipping the signs
 of $x^{(t)}$  we may assume that $|\{i \in S_t \mid x^{(t)}_i = -1\}| \geq \frac{m_t}{4}$. We update $w^{(t+1)}_i := (1+x^{(t)}_i) \cdot w^{(t)}_i$ (filling $x^{(t)}$ with zeros outside of $S_t$) so that the support $S_{t+1} = \textrm{supp}(w^{(t+1)})$ has size $|S_{t+1}| \leq \frac{3}{4} m_t$.
  Moreover, we have
  \[
    (1-s \varepsilon_t) Z_{w^{(t)}} \subseteq Z_{w^{(t+1)}} \subseteq (1+s\varepsilon_t) Z_{w^{(t)}}.
  \]
  Then iterating over $t=0,\ldots,T-1$ we have
  \[
   \prod_{t=0}^{T-1} (1-s \varepsilon_t) Z \subseteq Z_{w^{(T)}} \subseteq \prod_{t=0}^{T-1} (1+s \varepsilon_t) Z.
 \]
 Next, we use that the $\varepsilon_t$ are growing geometrically and so
 \[
   \prod_{t=0}^{T-1} (1+s \varepsilon_t) \leq \exp\Big( \sum_{t=0}^{T-1} s \varepsilon_t\Big)
   \leq \exp\Big(s\varepsilon_{T-1} \sum_{j \geq 0} (3/4)^{j/2}\Big) \leq \exp\Big(8\sqrt{\frac{c_0}{C}} s \varepsilon\Big) \leq 1+\varepsilon
 \]
 for $C$ large enough (depending on $s$ and $c_0$). Similarly one can bound $\prod_{t=0}^{T-1} (1-s \varepsilon_t) \geq 1-\varepsilon$.
We conclude that
\[
  (1-\varepsilon) Z \subseteq Z_{w^{(T)}} \subseteq (1+\varepsilon) Z,
\]
as claimed.

Finally, we discuss the running time aspect. All ingredients used in our argument can be implemented in polynomial time --- except the separation oracle for $Q_{\varepsilon_t}$ as required by Theorem~\ref{thm:FractionalColoring}. As $Q_{\varepsilon_t}$ is a symmetric convex body, separation is polynomial-time equivalent to testing membership, see \cite{GroetschelLovaszSchrijver1988}. That means we need to test the inclusion $Z_1 \subseteq Z_2$ polynomially many times, where $Z_1,Z_2 \subseteq \setR^n$ are zonotopes generated by at most $m$ segments each. While such a test is $\mathbf{coNP}$-complete in general~\cite{KulmburgAlthoff2021}, $Z_1$ has at most $2^m$ vertices that can be enumerated explicitly. Then for each vertex $x \in \textrm{vert}(Z_1)$ one can test whether $x \in Z_2$. Here we use that for a single zonotope, the separation problem is solvable in polynomial time by reducing it to the optimization problem.
\end{proof}

We may also summarize the algorithm behind Theorem~\ref{thm:L1SparsifierV2rephrased} as follows:
\begin{center}
\psframebox{
\begin{minipage}{13.2cm}
  {\sc $L_1$ Sparsification Algorithm} \vspace{1mm} \hrule \vspace{1mm}
{\bf Input:} Matrix $A \in \setR^{m \times n}$ generating zonotope $Z = \{ \sum_{i=1}^m y_ia_i : y \in [-1,1]^m \}$.
\begin{enumerate*}
\item[(1)] Set $w_i := 1$ for $i \in [m]$, $S := [m]$ and $M := \frac{Cn}{\eps^2}$
\item[(2)] WHILE $|S| > M$ DO
  \begin{enumerate*}
  \item[(3)] Let $Q := \{ x \in [-1,1]^{S} : (1-\tilde{\varepsilon}) Z_{w} \subseteq Z_{(\bm{1}+x) \odot w} \subseteq (1+\tilde{\varepsilon}) Z_{w}\}$
 where $\tilde{\eps}$ satisfies $|S| = \frac{c_0n}{\tilde{\eps}^2}$.
  \item[(4)] Compute $x \in sQ \cap [-1,1]^S$ such that at least $\frac{|S|}{4}$ entries $i \in S$ have $x_i=-1$.    
  \item[(5)] Update $w_i := w_i \cdot (1+x_i)$ for all $i \in S$ and $S := \textrm{supp}(w)$.
  \end{enumerate*}
\end{enumerate*}
\end{minipage}
}
\end{center}

\section{Linear-size spectral sparsifiers}

The goal of this section is to give an alternative proof of the spectral sparsification result of Batson, Spielman and Srivastava~\cite{BatsonSpielmanSrivastavaSTOC2009,BatsonSpielmanSrivastava2012} as stated in Theorem~\ref{thm:L2Sparsifier}.

For a symmetric matrix $B \in \setR^{n \times n}$ we define
\[
  \detplus(B) := \begin{cases} \det(B) & \textrm{if }B \succeq 0 \\ 0 & \textrm{otherwise.} \end{cases}
\]
Next we define the analogue of the function $f(t) = \Vol_n (K \ominus Z_t)$ that we had earlier.
\begin{lemma} \label{lem:PhiSepConvex}
Let $A_1, \dots, A_m \succeq 0$ and $B$ symmetric. Define a function $\Phi : \setR^m \to \setR_{\ge 0}$ by 
\[
        \Phi(t):= \detplus\big(B - \sum_{i=1}^m t_iA_i \big).
\]
Then $\Phi$ is separately convex.
\end{lemma}

\begin{proof}
  We prove the univariate case first: \\
  {\bf Claim I.} \emph{Let $A,B \in \setR^{n \times n}$ symmetric matrices with $A \succeq 0$. Then the function  $f : \setR \to \setR_{\ge 0}$ with
$
        f(s):= \det_+(B-sA)
$
is convex.} \\
{\bf Proof of Claim I.} If $f \equiv 0$, then $f$ is convex and there
is nothing to show.
Otherwise, after shifting (which affects $B$ but not $A$) we may assume that $f(0)>0$
and hence $B \succ 0$.
Let $\lambda_1 \ge \dots \ge \lambda_n\geq0$ be the eigenvalues of $B^{-1/2}AB^{-1/2}$. Then
\begin{eqnarray*}
  f(s) &=& \bm{1}_{B - sA \succeq 0} \cdot \det(B-sA) \\
  &=& \bm{1}_{s B^{-1/2}AB^{-1/2} \preceq I_n} \cdot \det(B) \cdot \det(I_n - s B^{-1/2}AB^{-1/2}) \\ &=& \mathbf{1}_{\{s \lambda_1 \leq 1\}} \cdot \det(B) \cdot \prod_{j=1}^n(1-s\lambda_j).
\end{eqnarray*}
For $s\lambda_1 \le 1$ we have $f''(s) =\det(B) \cdot \sum_{i\ne j}\lambda_i\lambda_j \prod_{\ell\ne i, j}(1-s\lambda_\ell)\ge 0$. Moreover, $f(s)$ is nonincreasing and vanishes when $s\lambda_1 =1$. Therefore $f$ is convex, as needed. \qed

Now back to the main proof. Let $t \in \setR^m$ and  $k \in [m]$.
Then apply Claim I with $B' :=B-\sum_{j\neq k}t_jA_j$ and $A' := A_k$ and the claim follows. 
\end{proof}
The reader may note that Claim I is the fact that a univariate real-rooted polynomial $p$ with $r \in \setR$ as smallest root is convex on $(-\infty,r]$, provided that the sign is chosen so that $p(s)>0$ for $s<r$.

Now we prove an analogue of Theorem~\ref{thm:RandomZonotopeInclusion}.
For matrices  $A_1,\dots,A_m\in\mathbb R^{n\times n}$ and $t \in \setR^m$ we abbreviate $A(t) := \sum_{i=1}^m t_iA_i$.
\begin{lemma} \label{lem:PSD-multiplier}
Let $A_1,\dots,A_m\in\mathbb R^{n\times n}$ be positive semidefinite matrices,
and let $B\succ0$. 
Let $X_1,\dots,X_m\geq0$ be independent random variables, and suppose $A(\E[X])\preceq B$. Then
\[
        \mathbb{P}[A(X)\preceq B] \ge \frac{\det(B-A(\E[X]))}{\det(B)}.
\]
\end{lemma}
\begin{proof}
If $A(X)\preceq B$ then $0\preceq B-A(X)\preceq B$ and $\det(B-A(X))\leq\det(B)$, so that we may bound $\Phi(X) \le \det(B)\cdot \mathbf 1_{\{A(X)\preceq B\}}$. Therefore, by Lemmas~\ref{lem:JensenForSepConvexFcts} and~\ref{lem:PhiSepConvex},
\[
        \det(B)\cdot \mathbb P[A(X)\preceq B] \ge \E[\Phi(X)] \ge \Phi(\E[X]) = \det(B-A(\E[X])). \qedhere
\]
\end{proof}
In particular if $A_1,\ldots,A_m \in \setR^{n \times n}$ are PSD matrices with sum $B := \sum_{i=1}^m A_i$, then
\begin{eqnarray} 
  \Pr_{x \sim [-1,1]^m}\Big[ \sum_{i=1}^m x_iA_i \preceq \varepsilon B\Big] &=& \Pr_{x \sim [-1,1]^m}\Big[ \sum_{i=1}^m (1+x_i)A_i \preceq (1+\varepsilon) B\Big] \label{eq:SimpleSpectralBound} \\
                                                                            &\stackrel{\textrm{Lem~\ref{lem:PSD-multiplier}}}{\geq}& \frac{\det( (1+\varepsilon)B - B)}{\det( (1+\varepsilon) B)} \nonumber
  \\ &=& \Big(\frac{\varepsilon}{1+\varepsilon}\Big)^n \nonumber
\end{eqnarray}
As in the case of zonotopes, this bound suffices to prove Theorem~\ref{thm:L2Sparsifier} with a slightly weaker bound of $O(\frac{n}{\varepsilon^2} \log(\frac{1}{\varepsilon}))$. Next, we work towards replacing 
\eqref{eq:SimpleSpectralBound} by a bound of the form $c^n$.

\subsection{A tight spectral positivity bound}

Now to the spectral positivity bound. Again, it will be convenient to first prove the bound for Rademacher random variables and then transfer the bound to other product distributions.
\begin{theorem} \label{thm:SpectralPositivityForRademachers}
  Let $A_1,\ldots,A_m \subseteq \setR^{n \times n}$ be matrices with $A_i \succeq 0$ for all $i \in [m]$. Then
  \[
 \Pr_{x \sim \{ -1,1\}^m}\Big[\sum_{i=1}^m x_iA_i \succeq 0\Big] \geq e^{-2n}
  \]
\end{theorem}

\begin{proof}
Let $V:=\mathrm{im}\Big(\sum_{i=1}^m A_i\Big)$ and $d:=\dim V$. Assume without loss of generality that $d>0$. Let $U \in \mathbb{R}^{n \times d}$ be a matrix whose columns are an orthonormal basis of $V$, so that $U^\top U=I_d$ and $UU^\top = P$ is the orthogonal projection onto $V$. Define $\widetilde A_j:=U^\top A_j U\in\setR^{d\times d}$ for each $j \in [m]$ and note that
\[
\sum_{i=1}^m x_iA_i\succeq0
\quad\Longleftrightarrow\quad
\sum_{i=1}^m x_i\widetilde A_i\succeq0.
\]
Moreover $\sum_{i=1}^m\widetilde A_i = U^\top\Big(\sum_{i=1}^m A_i\Big)U \succ0$. Thus, replacing $A_j$ by $\widetilde A_j$, it suffices to prove the result in dimension $d$, under the assumption that $\sum_{i=1}^m A_i\succ0$.

For \(x\in[-1,1]^m\), define $\Phi(x):=\detplus\Big(\sum_{i=1}^m x_iA_i\Big)$. By Lemma~\ref{lem:PhiSepConvex}, this function is separately convex, and for $x\in\{-1,1\}^m$ and $\rho\geq0$ we have  $\Phi(\rho x)=\rho^d\Phi(x)$. Moreover, $\Phi(\bm{1})=\det\Big(\sum_{i=1}^m A_i\Big)>0$.
Therefore Lemma~\ref{lem:hyperhomog} yields
\[
\Pr_{x\sim\{-1,1\}^m}[\Phi(x)>0]
\geq e^{-2d}.
\]
Finally, \(\Phi(x)>0\) implies that \(\sum_{i=1}^m x_iA_i\succeq0\), and so
\[
\Pr_{x\sim\{-1,1\}^m}
\Big[\sum_{i=1}^m x_iA_i\succeq0\Big]
\geq e^{-2d}\geq e^{-2n}. \qedhere
\]
\end{proof}

A random variable $X$ on $\setR$ is called \emph{symmetric} if $X$ and $-X$ have the same distribution.
We prove the following:

\begin{corollary} \label{cor:SpectralPositivityForSymDist}
  Let $A_1,\ldots,A_m \subseteq \setR^{n \times n}$ be matrices with $A_i \succeq 0$ for all $i \in [m]$. Then
  for any random vector $x \in \setR^m$ with independent symmetrically distributed coordinates one has
  \[
 \Pr\Big[\sum_{i=1}^m x_iA_i \succeq 0\Big] \geq e^{-2n}
  \]
  where $c>0$ is a universal constant.
\end{corollary}
\begin{proof}
By symmetry, we can draw any coordinate $x_i$ by first drawing the absolute value $\sigma_i \geq 0$ and then independently drawing a sign $\varepsilon_i \sim \{ -1,1\}$ so that $x_i = \varepsilon_i\sigma_i$. For any conditioning $\sigma$ we have
\[
 \Pr\Big[\sum_{i=1}^m \varepsilon_i \cdot (\sigma_iA_i) \succeq 0 \mid \sigma \Big] \geq e^{-2n}
\]
by Theorem~\ref{thm:SpectralPositivityForRademachers} which gives the claim.
\end{proof}

\subsection{From spectral inclusion to spectral sparsifiers}

\begin{corollary} \label{cor:VolumeSymmetricPSD}
There is a universal constant $c_0>0$ so that the following holds:
For any positive semidefinite matrices $A_1, \dots, A_m \in \setR^{n \times n}$ with sum $B := \sum_{i=1}^m A_i$ and $\eps > 0$ with $m \ge \frac{c_0n}{\eps^2}$, the set
\[
  Q = \Big\{ x \in [-1,1]^m : -\eps B \preceq \sum_{i=1}^m x_i A_i \preceq \eps B \Big\}
\]
has volume $\Vol_m (Q) \ge 2^{-5m}$.
\end{corollary}
\begin{proof}
  Let $K := \Big\{x \in [-1,1]^m : \sum_{i=1}^m x_iA_i \succeq -\eps B\Big\}$ so that $Q = K \cap -K$ and note that $[-\eps, \eps]^m \subseteq K$ as $A_i \succeq 0$ and $K$ is convex. 
  Fix a set $S \subseteq [m]$ and let $X \in [-1,1]^m$ be the random vector with independent coordinates so that
  $X_i \sim [-1,1]$ for $i \in S$ and $X_i = 0$ otherwise. Then the coordinates of $X$ are symmetrically
  distributed and hence
  \[\frac{\Vol_S (K_S)}{2^{|S|}} \geq \mathbb{P}\Big[\sum_{i=1}^m X_iA_i \succeq 0 \Big] \ge c^n
  \]
  for $c := e^{-2}$
  by Cor~\ref{cor:SpectralPositivityForSymDist}.
    For $c_0 \geq \log_2(1/c)$ we have $m \geq \frac{c_0n}{\varepsilon^2} \geq \frac{\log_2(c^{-n})}{\varepsilon^2}$ and so by Theorem~\ref{thm:LowerboundOnVolumeForKSymmetrizer} we have $\Vol_m(Q) \ge 2^{-5m}$.
\end{proof}

Now we can derive the following result which is the arbitrary-rank analogue of \cite{BatsonSpielmanSrivastavaSTOC2009,BatsonSpielmanSrivastava2012} due to De Carli, Harvey and Sato~\cite{SparseSumsOfPSDMatricesDeCarliHarveySato2015}:
\begin{theorem} \label{thm:PSDSparsifier}
  For any positive semidefinite matrices $A_1, \dots, A_m \in \setR^{n \times n}$ with $B = \sum_{i=1}^m A_i$ and any $\eps >0$ there exist weights $w_1, \dots, w_m \ge 0$ with $|\{i \in [m] : w_i \ne 0\}| \le O(\frac{n}{\eps^2})$ so that
  \[
    (1-\eps) B \preceq \sum_{i=1}^m w_iA_i \preceq (1+\eps) B
  \]
Moreover, the weights can be found in polynomial time.
\end{theorem}

\begin{proof}
  The proof is completely analogous to that of Theorem~\ref{thm:L1Sparsifier}, except here the membership test for the set $Q$ in Cor~\ref{cor:VolumeSymmetricPSD} boils down to the
  following: given symmetric matrices $A,B \in \setR^{n \times n}$, test if $A \preceq B$. This can be done in polynomial time.
\end{proof}

\begin{proof}[Proof of Theorem~\ref{thm:L2Sparsifier}] Let $A$ have rows $a_1, \dots, a_m$ and take $A_k = a_k a_k^\top$ in Theorem~\ref{thm:PSDSparsifier}. The matrix $D$ is formed by entries $D_{k,k} = \sqrt{w_k}$.
\end{proof}

\section{Acknowledgments}

The authors used GPT-5.5 Pro and GPT-5.6 Pro during the development of this work to explore proof strategies, search for related literature, and assist with verification. GPT was not used in any part of the exposition. 

\bibliographystyle{alphaurl} 
\bibliography{linearSizeL1Sparsifiers}

\appendix

\section{Concentration\label{sec:AppendixA}}

\begin{proof}[Proof of Lemma~\ref{lem:cubeconc}]
 Draw $z_k \sim \{-1,1\}$ independently for $k \in \mathbb{N}$, so $\mathbb E[e^{t z_k}] =\frac{e^t+e^{-t}}{2} \le e^{t^2/2}$. Let $z=\sum_{k \ge 1} 2^{-k} z_k$. Then $z \sim [-1,1]$ and, by independence,
\[
    \E [e^{tz}] =\prod_{k=1}^\infty \E[e^{t 2^{-k}z_k}] \le \prod_{k=1}^\infty \exp\Big(\frac{t^2 2^{-2k}}{2}\Big) =\exp\Big(\frac{t^2}{2}\sum_{k=1}^\infty 4^{-k}\Big)=e^{t^2/6}.
\]

Since $x \sim [-1,1]^m$, this gives $
    \displaystyle \E[e^{t \langle v, x \rangle}] = \prod_{k=1}^m \E[e^{tv_k x_k}] \le \exp\Big(\frac{t^2\|v\|_2^2}{6}\Big).
$ If $v = \mathbf{0}$ or $\lambda = 0$ there is nothing to prove; otherwise, taking $t = 3 \sqrt{\lambda}/\|v\|_2$,

\[
\begin{aligned}
\mathbb{P}\Big[\langle v,x\rangle > \sqrt{\lambda}\,\|v\|_2\Big]
&= \mathbb{P}\Big[e^{t\langle v,x\rangle}
    > e^{t\sqrt{\lambda}\|v\|_2}\Big] \\
&\le e^{-t\sqrt{\lambda}\|v\|_2}
    \E\Big[e^{t\langle v,x\rangle}\Big] \\
&\le \exp\Big(
    -t\sqrt{\lambda}\|v\|_2
    + \frac{t^2\|v\|_2^2}{6}
\Big) \\
&\le e^{-3\lambda/2}\\
& \le 2^{-2\lambda}. \qedhere
\end{aligned}
\]
\end{proof}

\end{document}